\documentclass[11pt]{article}
\usepackage{latexsym,amsmath,amsthm,verbatim,ifthen,amssymb}
\usepackage[latin1]{inputenc}
\usepackage[all]{xy}
\usepackage{graphicx}
\usepackage{epsfig}

\newdir{ >}{{}*!/-12pt/@{>}}
\newdir{ (}{!/-3pt/@_{(} } 

\newdir{ )}{!/-3pt/@^{(} } 

\newtheorem{theorem}{Theorem}[section]

\newtheorem{corollary}[theorem]{Corollary}
 \newtheorem{lemma}[theorem]{Lemma}
 \newtheorem{proposition}[theorem]{Proposition}
 \theoremstyle{definition}
 \newtheorem{definition}[theorem]{Definition}
 \theoremstyle{remark}
 \newtheorem{remark}[theorem]{Remark}
 
 \numberwithin{equation}{subsection}

\begin{document}

\title{A note on covers defining relative and sectional categories
\footnotetext{This work has been supported
by the Ministerio de Educaci\'on y Ciencia grant MTM2016-78647-P.}}

\author{J.M. Garc\'{\i}a-Calcines\footnote{Universidad de La Laguna,
Facultad de Matem\'aticas, Departamento de Matem\'aticas,
Estad\'{\i}stica e I.O., 38271 La Laguna, Spain. E-mail:
\texttt{jmgarcal@ull.es}} }

\maketitle

\begin{abstract}
We give a characterization of relative category in the sense of
Doe\-raene-El Haouari by using open covers. We also prove that
relative category and sectional category can be defined by taking
arbitrary covers (not necessarily open) when dealing with ANR
spaces.
\end{abstract}

\vspace{0.5cm}
\noindent{2010 \textit{Mathematics Subject Classification} : 55M30, 55M15, 55P99.}\\
\noindent{\textit{Keywords} : Relative category, sectional
category, Lusternik-Schnirelmann category, topological complexity}
\vspace{0.2cm}

\section*{Introduction}

The aim in this paper is twofold. On the one hand we establish a
characterization of relative category by means of open covers.
Relative category of a map $f$, written $\mbox{relcat}(f),$ was
introduced in \cite{D-H} by J.P. Doeraene and M. El Haouari in
order to give a relative version of sectional category \cite{Sch}
(also called Schwarz genus). However, unlike sectional category,
relative category is not defined by means of open covers but,
instead, directly by the join construction, that is, the homotopy
pullback of the homotopy pushout of two maps having the same
target. However, considering in this paper covers by
\emph{relatively sectional} open subsets in a given cofibration
$i_X:A\hookrightarrow X$ we can establish another numerical
homotopy invariant, denoted as $\mbox{relcat}^{op}(i_X)$ in such a
way it agrees with $\mbox{relcat}(i_X)$ under non very restrictive
conditions. In this sense, in the first section of this paper we
prove the following theorem:

\medskip
\textbf{Theorem.} Let $i_X:A\hookrightarrow X$ be a cofibration
where $X$ is a normal space. Then
$\mbox{relcat}(i_X)=\mbox{relcat}^{op}(i_X).$
\medskip

As a consequence of this result we are able to prove that
Fadell-Huseini's relative category \cite{F-H} is an upper bound of
Doeraene-El Haouari's relative category. We can also obtain a
certain cohomological lower bound.

On the other hand, our second objective in this paper is to
establish generalized versions to both sectional category and
relative category. Indeed, in \cite{Sr,Sr2} T. Srinivasan defined
the notion of generalized Lus\-ter\-nik-Schni\-rel\-mann category
of a space $X,$ $\mbox{cat}_g(X),$ as the least nonnegative
integer $n$ such that $X$ admits a cover constituted by $n+1$
subsets (not necessarily open) which are contractible in $X.$ It
is a numerical homotopy invariant which agrees with
$\mbox{cat}(X)$, the usual notion of Lusternik-Schnirelmann
category, provided $X$ is an ANR. Following the same spirit as in
Srinivasan's work, in the second section of the paper we first
extend this notion to sectional category, with the so-called
\emph{generalized sectional category}, denoted as
$\mbox{secat}_g(-),$ establishing that it agrees with the usual
sectional category when dealing with ANR spaces:

\medskip
\textbf{Theorem.} Let $p:E\rightarrow B$ be a fibration between
ANR spaces. Then $\mbox{secat}_g(p)=\mbox{secat}(p).$
\medskip

This result has a nice consequence involving the \emph{generalized
topological complexity}. In this sense, when $X$ has the homotopy
type of a CW-complex, it is proved that the topological complexity
of $X$ may be defined by using a general cover of $X\times X,$ not
necessarily open. Another consequence is that we can recover T.
Srinivasan's result $\mbox{cat}(X)=\mbox{cat}_g(X)$ when $X$ is an
ANR.

Finally, using the open cover characterization of relative
category given in the first section we also introduce a
generalized version of $\mbox{relcat}(-)$, denoted as
$\mbox{relcat}_g(-),$ proving:

\medskip
\textbf{Theorem.} Let $i_X:A\hookrightarrow X$ be a cofibration
between ANR spaces. Then
$\mbox{relcat}(i_X)=\mbox{relcat}_g(i_X).$
\medskip

As a consequence, we can compare the monoidal topological
complexity with its generalized counterpart, establishing that
they agree whenever we are considering ANR spaces.

\section{Doeraene-El Haouari's relative category defined by open covers}

We begin in this section by recalling the notion of relative
category of a map \cite{D-H}. As we have commented in the
introduction, relative category is defined by using the join
construction, in particular Ganea maps. In general, the join of
two maps $f:X\rightarrow Z$ y $g:Y\rightarrow Z$, $X*_Z Y,$ is
given by the homotopy pushout of the homotopy pullback of $f$ and
$g:$
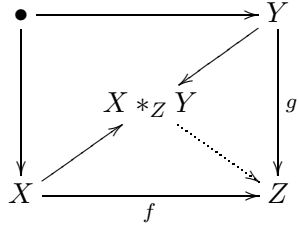
$$\xymatrix@C=0.7cm@R=0.7cm{ {\bullet } \ar[rr] \ar[dd] & & {Y} \ar[dl] \ar[dd]^g \\
 & {X*_Z Y} \ar@{.>}[dr] & \\ {X} \ar[ur] \ar[rr]_f & & {Z} }$$
\noindent where the dotted arrow, called join map, is the
corresponding co-whisker map, which is induced by the weak
universal property of the homotopy pushout. Now, if $i
_X:A\rightarrow X$ is a map, then $n$-th Ganea map of $i_X$,
$g_n:G_n(X)\rightarrow X,$ is the join map inductively defined in
the following join construction ($n\geq 0$):
$$
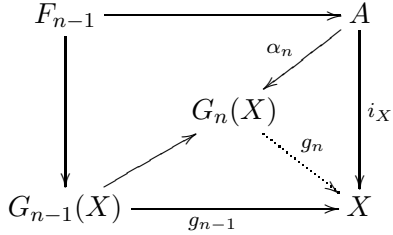
\xymatrix@C=0.7cm@R=0.7cm{ {F_{n-1} } \ar[rr] \ar[dd] & & {A} \ar[dl]_{\alpha _n} \ar[dd]^{i_X} \\
 & {G_{n}(X)} \ar@{.>}[dr]^{g_{n}} & \\ {G_{n-1}(X)} \ar[ur] \ar[rr]_{g_{n-1}} & & {X} }$$
\noindent where $g_0:=i_X:A\rightarrow X.$ Then the \emph{relative
category} of $i_X$, denoted $\mbox{relcat}(i_X)$, is defined as
the least nonnegative integer $n$ such that $g_n:G_n(X)\rightarrow
X$ admits a homotopy section $\sigma :X\rightarrow G_n(X)$
satisfying that $\sigma \circ i_X\simeq \alpha _n.$ It is
important to note that relative category has a Whitehead
characterization. Indeed, for all $n$ we can also consider the
$n$-th sectional fat-wedge $t_n:T^n(i_X)\rightarrow X^{n+1}$
inductively defined as follows: For $n=0$ set $T^0(i_X):=A$ and
$t_0=i_X:A\rightarrow X$. If $t_{n-1}:T^{n-1}(i_X)\rightarrow X^n$
is already given, then $t^n$ is the join map
$$\xymatrix@C=0.5cm@R=0.6cm{ {\bullet } \ar[rr] \ar[dd] & & {X^n\times A}
\ar[dl] \ar[dd]^{1_{X^n}\times i_X} \\
 & {T^n(i_X)} \ar@{.>}[dr]^{t_n} & \\
 {T^{n-1}(i_X)\times X} \ar[ur] \ar[rr]_{t_{n-1}\times 1_X} & & {X^{n+1}} }$$

Moreover, there exists a homotopy commutative diagram of the form:
$$\xymatrix{
{A} \ar[rr]^{\tau _n} \ar[d]_{i_X} & & {T^n(i_X)}
\ar[d]^{t_n} \\
{X} \ar[rr]_{\Delta _{n+1}} & & {X^{n+1}}   }$$ J.P. Doeraene and
El Haouari prove the following characterization of relative
category:

\begin{theorem}\label{DH}\cite[Prop. 26]{D-H}
Let $i_X:A\rightarrow X$ be a map. Then $\mbox{relcat}(i_X)\leq n$
if and only if there exists a map $f:X\rightarrow T^n(i_X)$ making
commutative, up to homotopy, the following diagram
$$\xymatrix{
{A} \ar[rr]^{\tau _n} \ar[d]_{i_X} & & {T^n(i_X)}
\ar[d]^{t_n} \\
{X} \ar@{.>}[urr]^f  \ar[rr]_{\Delta _{n+1}} & & {X^{n+1}}   }$$
\end{theorem}

For our purposes of establishing an open cover characterization of
relative category we are specially interested in cofibrations.
Throughout this paper by a cofibration we will mean \emph{a closed
map having the homotopy extension property}. The main reason of
using cofibrations is that, in this case, we have a nice explicit
description of the $n$-th sectional fat-wedge:

\begin{proposition}\cite[Cor. 11]{C-V} {\rm
Let $i_X:A\hookrightarrow X$ be a cofibration. Then the $n$-th
sectional fat-wedge $t_n:T^n(i_X)\hookrightarrow X^{n+1}$ is, up
to homotopy equivalence, the following subspace of $X^{n+1}$:
$$T^n(i_X)=\{(x_0,x_1,...,x_n)\in X^{n+1}\hspace{3pt}:x_i\in A\hspace{3pt}\mbox{for some}\hspace{3pt}i\}$$
\noindent $t_n$ being the canonical inclusion ($t_n$ is actually,
a cofibration).}
\end{proposition}

Moreover, one can easily check that $\tau _n:A\rightarrow
T^n(i_X)$ is defined, up to homotopy equivalence, as $\tau
_n(a)=(a,a,...,a).$ If $\Delta _{n+1}:X\rightarrow X^{n+1}$
denotes the $(n+1)$-th diagonal map, then there exists a strictly
commutative diagram:
$$\xymatrix{
{A} \ar[rr]^{\tau _n} \ar@{^{(}->}[d]_{i_X} & & {T^n(i_X)}
\ar@{^{(}->}[d]^{t_n} \\
{X} \ar[rr]_{\Delta _{n+1}} & & {X^{n+1}}   }$$

Now we give our definition of relative category by means of open
covers. As we have previously said, we will consider cofibrations.

\begin{definition}\label{primordial}{\rm
Let $i_X:A\hookrightarrow X$ be a cofibration. We say that a
subset $U\subseteq X$ is \emph{relatively sectional} if
$A\subseteq U$ and there exists a homotopy of pairs $H:(U\times
I,A\times I)\rightarrow (X,A)$ satisfying that $H(x,0)=x$ and
$H(x,1)\in A$ for all $x\in U.$  Then we define
$\mbox{relcat}^{op}(i_X)$ as the least nonnegative integer $n$
such that $X$ admits a cover constituted by $n+1$ relatively
sectional open subsets. If such an integer does not exist then we
simply set $\mbox{relcat}^{op}(i_X)=\infty .$}
\end{definition}

\begin{remark}
We can obviously define $\mbox{relcat}^{op}(i_X)$ when $i_X$ is a
general map by considering the cofibration that approximates
$i_X.$ However, for the sake of simplicity, in this paper we have
preferred to work directly with cofibrations.
\end{remark}

Now we see the relationship between $\mbox{relcat}^{op}$ and
relative category in the sense of Doeraene-El Haouari. For this
task we will use the following lemma, whose proof can be found,
for instance, in \cite{War}, and therefore is omitted:

\begin{lemma}\label{previo}{\rm
Let $j:A\hookrightarrow X$ be a cofibration and $f:X\rightarrow X$
a map such that $f\circ j=j$ and $f\simeq 1_X.$ Then there exists
a map $g:X\rightarrow X$ satisfying $g\circ j=j$ and $g\circ
f\simeq 1_X$ rel. $A.$}
\end{lemma}

Now we are in conditions to state and prove the main result in
this section.

\begin{theorem}\label{mainth}{\rm
If $i_X:A\hookrightarrow X$ is a cofibration where $X$ is a normal
space, then $\mbox{relcat}(i_X)=\mbox{relcat}^{op}(i_X).$ }
\end{theorem}

\begin{proof}

Suppose that $\mbox{relcat}^{op}(i_X)\leq n$ and take
$\{U_i\}_{i=0}^n$ a cover of $X$ by relatively sectional open
subsets. Also consider homotopies $H_i:(U_i\times I,A\times
I)\rightarrow (X,A)$ such that $H_i(x,0)=x$ and $H_i(x,1)\in A.$
As $X$ is normal we can take closed subsets $A_i,B_i$ and open
subsets $\Theta _i$ satisfying that $A_i\subseteq \Theta
_i\subseteq B_i\subseteq U_i$ where $\{A_i\}_{i=0}^n$ is a cover
of $X.$ There are also Urysohn maps $h_i:X\rightarrow [0,1]$ such
that $h_i(A_i)=\{1\}$ and $h_i(X\setminus \Theta _i)=\{0\}.$ We
define $L_i:(X\times I,A\times I)\rightarrow (X,A)$ as
$$L_i(x,t)=\begin{cases}
x, & x\in X\setminus B_i \\
H_i(x,th_i(x)), & x\in U_i
\end{cases}$$
\noindent and take $L:(X\times I,A\times I)\rightarrow
(X^{n+1},T^n(i_X))$ as $L=(L_0,L_1,...,L_n).$ Defining
$f(x):=L(x,1)$, we obtain a map $f:X\rightarrow T^n(i_X)$
satisfying $L|_{A\times I}:\tau _n\simeq f\circ i_X$ and $L:\Delta
_{n+1}\simeq t_n\circ f$. Therefore, by Theorem \ref{DH},
$\mbox{relcat}(i_X)\leq n.$

Conversely, suppose that $\mbox{relcat}(i_X)\leq n$ and take,
again by Theorem \ref{DH}, a map $f:X\rightarrow T^n(i_X)$
satisfying $f\circ i_X\simeq \tau _n$ and $t_n\circ f\simeq \Delta
_{n+1}.$ As $i_X$ is a cofibration we can suppose, without losing
generality, that $f\circ i_X=\tau _n$ and $t_n\circ f\simeq \Delta
_{n+1}.$ Take a homotopy $L: t_n\circ f\simeq \Delta _{n+1}$ and
consider $t_n\circ f=(f_0,...,f_n)$ and $L=(L_0,...,L_n)$ with
$f_i:X\rightarrow X$ and $L_i:X\times I\rightarrow X$ for all
$i\in \{0,1,...,n\}.$ Observe that $f_i \circ i_X=i_X$ and
$L_i:f_i\simeq 1_X.$ Hence, by Lemma \ref{previo} above, we can
find a map $g_i:X\rightarrow X$ such that $g_i \circ i_X=i_X$ and
a homotopy $L'_i:1_X\simeq g_i\circ f_i $ rel $A.$ We establish
$$\varphi _i:=g_i\circ f_i:X\rightarrow X$$
Taking into account that $(\varphi _0(x),...,\varphi _n(x))\in
T^n(i_X)$ for all $x\in X,$ we obtain $\varphi :X\rightarrow
T^n(i_X)$ such that $t_n\circ \varphi =(\varphi _0,...,\varphi
_n).$ Obviously, $\varphi \circ i_X=\tau _n$ and
$L'=(L'_0,...,L'_n)$ is a homotopy $L':\Delta _{n+1}\simeq
t_n\circ \varphi $ rel $A.$

Now, using the fact that $i_X$ is a cofibration, there exists
$A\subseteq U\subseteq X,$ an open neighbourhood of $A$ such that
$A$ is a strong deformation rectract of $U$. Take $r:U\rightarrow
A$ the corresponding retraction and $H:U\times I\rightarrow X$ a
homotopy verifying $H(x,0)=x,$ $H(x,1)=r(x)$ and $H(a,t)=a.$
Defining $U_i:=\varphi _i^{-1}(U)$ we have that $\{U_i\}_{i=0}^n$
is an open cover of $X$ in such a way that each $U_i$ contains
$A.$ Finally, the homotopy $G_i:(U_i\times I,A\times I)\rightarrow
(X,A)$ defined as
$$G_i(x,t)=\begin{cases}
L'_i(x,2t), & 0\leq t\leq \frac{1}{2} \\
H(\varphi _i(x),2t-1), & \frac{1}{2}\leq t\leq 1
\end{cases}$$
\noindent gives $G_i(x,0)=x$ and $G_i(x,1)\in A.$ This means that
each  $U_i$ is relatively sectional and therefore
$\mbox{relcat}^{op}(i_X)\leq n.$
\end{proof}

In actual fact, taking a careful reading of the proof above, we
have the following characterizations of relative category:

\begin{corollary}\label{carac}
Let $i_X:A\hookrightarrow X$ be a cofibration, where $X$ is a
normal space. Then the following statements are equivalent:
\begin{enumerate}
\item[(i)] $\mbox{relcat}(i_X)\leq n$.

\item[(ii)] $\mbox{relcat}^{op}(i_X)\leq n$.

\item[(iii)] $X$ admits an open cover $U_0,U_1,...,U_n$
such that for all $i$, $A\subseteq U_i$ and there exists a
homotopy $H_i:U_i\times I\rightarrow X$ verifying $H_i(x,0)=x,$
$H_i(x,1)\in A$ and $H_i(a,t)=a$, for all $x\in X,a\in A$ and
$t\in I.$

\item[(iv)] $X$ admits an open cover $U_0,U_1,...,U_n$
such that for all $i$, $A\subseteq U_i$ and there exists a
homotopy $H_i:U_i\times I\rightarrow X$ verifying $H_i(x,0)=x,$
$H_i(x,1)\in A$ and $H_i(a,1)=a$, for all $x\in X,a\in A$.
\end{enumerate}
\end{corollary}

\begin{proof}
By the theorem above, (i) and (ii) are equivalent. From the second
part of the proof of the theorem we deduce that both parts are
also equivalent to (iii). Moreover, (iii) clearly implies (iv).
Finally, if $X$ admits an open cover as in part (iv), then we can
consider closed subsets $A_i,B_i$ and open subsets $\Theta _i$
satisfying $A\subseteq A_i\subseteq \Theta _i\subseteq
B_i\subseteq U_i$ and $\{A_i\}_{i=0}^n$ covers $X.$ A similar
argument to that done in the firts part of the proof in the
theorem above shows that $\mbox{relcat}(i_X)\leq n.$
\end{proof}

An interesting consequence we obtain of this characterization is
the relationship between Doeraene-El Haouari's relative category
and Fadell-Husseini's relative category \cite{F-H} of a cofibred
pair $(X,A).$ Recall that given $A\hookrightarrow X$ a
cofibration, the Fadell-Husseini relative category of $(X,A)$,
denoted as $\mbox{cat}^{FH}(X,A),$ is the least nonnegative
integer $n$ such that $X$ admits a cover constituted by $n+1$ open
subsets $\{U_i\}_{i=0}^n$ satisfying:
\begin{enumerate}
\item[(i)] $A\subseteq U_0$ and there exists a homotopy of pairs
$H:(U_0\times I,A\times I)\rightarrow (X,A)$ such that $H(x,0)=x$
and $H(x,1)\in A$ for all $x\in U_0$ (that is, $U_0$ is a
relatively sectional open subset).
\item[(ii)] For all $i\geq 1,$ $U_i$ is contractible in $X$.
\end{enumerate}

We obtain the following inequality:

\begin{proposition}{\rm
Let $i_X:A\hookrightarrow X$ be a cofibration where $X$ is normal
and path-connected. Then, $\mbox{relcat}(i_X)\leq
\mbox{cat}^{FH}(X,A).$}
\end{proposition}

\begin{proof}
Suppose that $\mbox{cat}_{FH}(X,A)=n$ and consider
$\{U_i\}_{i=0}^n$ an open cover $X$ such that:
\begin{enumerate}
\item[(i)] $A\subseteq U_0$ and there exists a homotopy of pairs
$H:(U_0\times I,A\times I)\rightarrow (X,A)$ such that $H(x,0)=x$
and $H(x,1)\in A$ for all $x\in U_0.$
\item[(ii)] For all $i\geq 1,$ $U_i$ is contractible in $X$.
\end{enumerate}

As $X$ is path-connected we can take a homotopy $H_i:U_i\times
I\rightarrow X$ with $H_i(x,0)=0$ and $H_i(x,1)=a_0\in A$ for all
$x\in U_i$ and $i\geq 1.$ We can also suppose, without losing
generality, that $U_i\cap A=\emptyset $ for all $i\geq 1$
(otherwise we take the open cover $\{U_0,U_1\setminus
A,...,U_n\setminus A\}$ of $X$). By normality, we take an open
cover $\{V_i\}_{i=0}^n$ satisfying $V_i\subseteq
\overline{V_i}\subseteq U_i$ for all $i\geq 0.$ Then we can define
$$\Omega :=U_0\cap (X\setminus \overline{V_1})\cap...\cap
(X\setminus \overline{V_n}))$$ \noindent which is an open
neighborhood of $A$ such that $\Omega \cap V_i=\emptyset $ for all
$i\geq 1.$ If $U'_0:=U_0$ and $U'_i:=\Omega \cup V_i $ for $i\geq
1,$ then it is not difficult to check that $\{U'_i\}_{i=0}^n$ is
an open cover of $X.$ Moreover, for $i\geq 1$ the homotopy
$H'_i:(U'_i\times I,A\times I)\rightarrow (X,A)$ defined as
$$H'_i(x,t)=\begin{cases}H(x,t), & x\in \Omega \\
H_i(x,t), & x\in V_i
\end{cases}$$
\noindent satisfies $H'_i(x,0)=x$ and $H'_i(x,1)\in A.$ This fact
implies that each $U'_i$ is a relatively sectional open subset and
therefore $\mbox{relcat}(i_X)=\mbox{relcat}^{op}(i_X)\leq n.$
\end{proof}

\begin{remark}
Observe that, actually, the proof above also works when the pair
$(X,A)$ is 0-connected, instead of requiring path-connectedness
for $X.$
\end{remark}

Another interesting consequence is given by the nilpotence of the
relative cohomology.

\begin{proposition}{\rm
If $i_X:A\hookrightarrow X$ is a cofibration with $X$ a normal
space, then $\mbox{nil}\hspace{2pt}H^*(X,A)\leq
\mbox{relcat}(i_X).$}
\end{proposition}

\begin{proof}
Suppose $\mbox{relcat}(i_X)=n$ and take $\{U_i\}_{i=0}^n$ a cover
of $X$ constituted by $n+1$ relatively sectional open subsets, and
consider $x_0,x_1,...,x_n\in H^*(X,A).$ For all $i$, if $x_i\in
H^{n_i}(X,A),$ also consider the following part of the long exact
sequence in cohomology of the triple $(X,U_i,A)$
$$...{\longrightarrow }H^{n_i}(X,U_i)\stackrel{q_i^*}{\longrightarrow }
H^{n_i}(X,A)\stackrel{k_i^*}{\longrightarrow }
H^{n_i}(U_i,A){\longrightarrow }...$$ \noindent where
$q_i:(X,A)\hookrightarrow (X,U_i)$ and $k_i:(U_i,A)\hookrightarrow
(X,A)$ denote the respective inclusions. Taking into account the
existence of homotopy (of pairs) commutative diagrams of the form
$$\xymatrix{
{(U_i,A)} \ar@{^{(}->}[rr]^{k_i} \ar[dr] & & {(X,A)} \\
 & {(A,A)} \ar@{^{(}->}[ur] &
}$$ \noindent we have that $k_i^*=0$ so each $q_i^*$ is surjective
and we can pick $\bar{x}_i\in H^{n_i}(X,U_i)$ such that
$q_i^*(\bar{x}_i)=x_i$ for all $i\geq 0.$ But then
$$x_0\cup x_1\cup ...\cup x_n=q_0^*(\bar{x}_0)\cup q_1^*(\bar{x}_1)\cup ...\cup q_n^*(\bar{x}_n))
=q^*(\bar{x}_0\cup \bar{x}_1\cup ...\cup \bar{x}_n)$$ \noindent
where $q:(X,A)\hookrightarrow (X,\bigcup _{i=0}^nU_i)$ denotes the
natural inclusion. But $\{U_i\}_{i=0}^n$ is a cover of $X$, so
$q^*=0$ and $x_0\cup x_1\cup ...\cup x_n=0.$
\end{proof}

\begin{remark}
We point out that
$$\mbox{nil}\hspace{2pt}H^*(X,A)=\mbox{nil}\hspace{2pt}H^*(X/A)=
\mbox{cuplength}\hspace{2pt}(X/A)$$ On the other hand, there
exists a similar lower bound for  $\mbox{secat}(i_X).$ Indeed, it
is well-known that, if $i_X^*:H^*(X)\rightarrow H^*(A)$ denotes
the homomorphism induced in cohomology, then
$\mbox{nil}\hspace{2pt}\mbox{ker}\hspace{2pt}(i_X^*)\leq
\mbox{secat}(i_X)$. It is natural to ask whether these two notions
$\mbox{nil}\hspace{2pt}\mbox{ker}\hspace{2pt}(i_X^*)$ and
$\mbox{nil}\hspace{2pt}H^*(X,A),$ are related or not. In
\cite[Th.21 (e)]{C-V} we found an answer to this question: If
$i_X$ has a homotopy retraction, then
$\mbox{nil}\hspace{2pt}\mbox{ker}\hspace{2pt}(i_X^*)=\mbox{cuplength}\hspace{2pt}(X/A)=
\mbox{nil}\hspace{2pt}H^*(X,A).$

\end{remark}

\section{General covers}

As we have previously commented, T. Srinivasan defined in
\cite{Sr,Sr2} the notion of generalized
Lus\-ter\-nik-Schni\-rel\-mann category of a space $X,$
$\mbox{cat}_g(X),$ as the least nonnegative integer $n$ (or
infinity) such that $X$ admits a cover constituted by $n+1$
subsets (not necessarily open) which are contractible in $X.$  It
is a numerical homotopy invariant which agrees with
$\mbox{cat}(X)$ when $X$ is an ANR. Our aim in this section is to
extend this notion to both sectional and relative categories.

\subsection{Generalized sectional category}

Recall that the sectional category of a fibration $p:E\rightarrow
B$, $\mbox{secat}(p),$ is defined as the least nonnegative integer
$n$ (or infinity) such that $B$ admits a cover constituted by
$n+1$ open subsets on each of which there exists a continuous
local section of $p$. This is a numerical homotopy invariant.
Requiring homotopy section instead of section permits one to
extend the notion of sectional category to any map. This is a
variant of Lusternik-Schnirelmann category (or L.S category, for
short) and also a generalization, since
$\mbox{secat}(p)=\mbox{cat}(B)$ when $E$ is a contractible space.
The sectional category was introduced by Schwarz \cite{Sch}, under
the name genus and then renamed by James \cite{J}.

We begin this subsection by establishing the notion of
\emph{generalized sectional category}.

\begin{definition}
Let $p:E\rightarrow B$ be a fibration. The \emph{generalized
sectional category of $p$}, denoted as $\mbox{secat}_g(p),$ is the
least nonnegative integer $n$ (or $\infty $) such that $B$ admits
a cover by $n+1$ subsets $A_0,...,A_n$ on each of which there
exists a continuous local section of $p$, that is, a commutative
diagram of the form
$$\xymatrix{
{A_i} \ar@{^(->}[rr] \ar[dr]_{s_i} & & B \\ & E \ar[ur]_p }$$
\end{definition}

We clearly have the inequality $\mbox{secat}_g(p)\leq
\mbox{secat}(p).$ In addition, it is possible to extend this
definition to any map $f:E\rightarrow B$, not necessarily a
fibration, by just requiring continuous local sections, up to
homotopy. When $f$ is a fibration both notions agree.

Generalized sectional category is a homotopy invariant in the next
sense. The proof of the result we are going to state is omitted as
it is just a simple repetition of the classical case.

\begin{proposition}\label{invariance}
Consider the following homotopy commutative diagram, where $\alpha
$ and $\beta $ are homotopy equivalences:
$$\xymatrix{
{X} \ar[r]^{\alpha }_{\simeq } \ar[d]_f & {X'} \ar[d]^{f'} \\
{Y} \ar[r]^{\simeq }_{\beta } & {Y'} }$$ Then
$\mbox{secat}_g(f)=\mbox{secat}_g(f').$
\end{proposition}

The first example of generalized sectional category is the
generalized L.-S. category.  As in the classical case, in order to
establish $\mbox{cat}_g(X)$ when $X$ is a path-connected space
with base point $x_0\in X$, we can suppose that every subset $A_i$
is contractible in $X$ to the point $x_0.$ This means that the
inclusion map $A_i\hookrightarrow X$ is homotopic to the constant
map at $x_0.$ Therefore, whenever we have a path-connected pointed
space $X$, $\mbox{cat}_g(X)$ can be defined as
$$\mbox{cat}_g(X)=\mbox{secat}_g(*\stackrel{x_0}{\longrightarrow }X)$$
\noindent where $x_0:*\rightarrow X$ denotes the map carrying $*$
to $x_0.$ In particular, taking the path fibration
$p:P(X)\rightarrow X$, where $P(X)=\{\alpha \in X^I:\alpha
(0)=x_0\}$ and $p(\alpha )=\alpha (1)$, we have by Proposition
\ref{invariance} that
$$\mbox{cat}_g(X)=\mbox{secat}_g(p:P(X)\rightarrow X)$$

A basic property, relating the generalized versions of sectional
category with L.-S. category, is given in the next proposition.
Observe that in any homotopy commutative diagram of the form
$$\xymatrix{
{E} \ar[rr]^{\alpha } \ar[dr]_f & & {E'} \ar[dl]^{f'} \\  & B &}$$
\noindent we have the inequality $\mbox{secat}_g(f')\leq
\mbox{secat}_g(f).$ Indeed, if $A$ is a subset of $B$ together
with $s:A\rightarrow E$ a local homotopy section of $f$, then, we
clearly have that $s':=\alpha \circ s:A\rightarrow E'$ is a local
homotopy section of $f'.$

\begin{proposition}\label{desig}
Let $f:X\rightarrow Y$ be a map where $Y$ is a path-connected
based space. Then $\mbox{secat}_g(f)\leq \mbox{cat}_g(Y).$
Moreover, if $X$ is contractible, then $\mbox{secat}_g(f)=
\mbox{cat}_g(Y).$
\end{proposition}

\begin{proof}
By the path-connectedness of $Y$, the following diagram is
commutative, up to homotopy:
$$\xymatrix{ {*} \ar[rr] \ar[dr]_{y_0} & & {X} \ar[dl]^{f} \\
& Y &}$$ \noindent so we can apply the argument above. Also
observe that whenever $X$ is contractible, $*\rightarrow X$ is a
homotopy equivalence so we can apply Proposition \ref{invariance}.
\end{proof}

Another interesting example of generalized sectional category is
the \emph{generalized topological complexity} of a space $X$. It
is simply defined as
$$\mbox{TC}_g(X):=\mbox{secat}_g(\pi _X:X^I\rightarrow X\times X)$$
\noindent where $\pi _X$ denotes the path fibration, given as $\pi
_X(\alpha )=(\alpha (0),\alpha (1)).$ Again, as $\pi _X$ is the
fibration that approximates the diagonal map $\Delta
_X:X\rightarrow X\times X,$ by Proposition \ref{invariance}, we
have $\mbox{TC}_g(X)=\mbox{secat}_g(\Delta _X:X\rightarrow X\times
X).$ Also note by the same proposition that $\mbox{TC}_g$ is a
homotopy invariant. Moreover:

\begin{proposition}
Let $X$ be a path-connected space. Then
$$\mbox{cat}_g(X)\leq \mbox{TC}_g(X)\leq \mbox{cat}_g(X\times X)$$
\end{proposition}

\begin{proof}
By Proposition \ref{desig}, $\mbox{TC}_g(X)\leq
\mbox{cat}_g(X\times X)$ so it only remains to prove the
inequality $\mbox{cat}_g(X)\leq \mbox{TC}_g(X).$ In order to prove
this one just has to follow the same proof as in the classical
case (see \cite{F}).
\end{proof}

Now we prove that generalized sectional category agrees with
sectional category of a fibration provided we are dealing with
absolute neighborhood retracts (for metric spaces). Recall that an
absolute neighborhood retract, or ANR, is a metric space $Y$ such
that for any metric space $X$ and closed subset $A$, any
continuous map $f:A\rightarrow Y$ has an extension
$\tilde{f}:U\rightarrow Y$ for some neighborhood $U$ of
$A\subseteq X.$

In order to fulfill our purposes we first recall two crucial
results about ANR spaces. Such results are also used in \cite{Sr,
Sr2}. Given $\mathcal{U}=\{U_{\lambda }\}_{\lambda \in \Lambda }$
a cover of a space $Y$, we say that two maps $f,g:X\rightarrow Y$
are \emph{$\mathcal{U}$-close}, if for any $x\in X$ there exists
$\lambda \in \Lambda $ such that $f(x)$ and $g(x)$ belong to
$U_{\lambda }.$

\begin{lemma}\cite{Hu}\label{epsilon}
Let $Y$ be an ANR space. Then there exists a function (not
necessarily continuous) $\varepsilon :Y\rightarrow (0,+\infty )$
such that the cover
$$\mathcal{V}=\{B(y,\varepsilon (y)):y\in Y\}$$ \noindent
satisfies that any pair of  $\mathcal{V}$-close maps
$f,g:X\rightarrow Y$ are homotopic.
\end{lemma}

The other result we need is a version of a lemma that appears in
\cite{W}. Observe that our statement is slightly different to the
one given in T. Srinivasan's work. However, its proof only needs
very small straightforward modifications of the one given in
\cite{Sr2}. This way it actually consists of an almost identical
argument so we have decided to omit it.


\begin{lemma}\label{walsh}
Let $X$ be a metric space, $A\subseteq X$, $K$ an ANR space and
$f:A\rightarrow K$ a map. Then, for any function $\varepsilon
:K\rightarrow (0,+\infty )$ (not necessarily continuous) there
exist an open neighborhood  $U\supseteq A$ in $X$ and a map
$g:U\rightarrow K$ such that:
\begin{enumerate}
\item[(i)] For all $u\in U$ there is $a_u\in A$ satisfying
$$\max \{d(g(u),f(a_u)),d(u,a_u))\}<\varepsilon (f(a_u))$$

\item[(ii)] $g_{|A}\simeq f.$
\end{enumerate}
\end{lemma}

Now we are in conditions to state and prove one of the main
results of this section.

\begin{theorem}\label{chulo}
Let $p:E\rightarrow B$ be a fibration where $E$ and $B$ are ANR
spaces. Then $\mbox{secat}_g(p)=\mbox{secat}(p).$
\end{theorem}

\begin{proof}
It only remains to check the inequality $\mbox{secat}(p)\leq
\mbox{secat}_g(p)$. As $B$ is an ANR, by Lemma \ref{epsilon} there
exists a function $\varepsilon :B\rightarrow (0,+\infty )$ such
that the cover $\mathcal{W}=\{B(b,\varepsilon (b)):b\in B\}$ of
$B$ satisfies that any pair of $\mathcal{W}$-close maps
$f,g:Z\rightarrow B$ are homotopic. By continuity of
$p:E\rightarrow B$, for each $e\in E$ pick $\delta (e)>0$ such
that $p(B(e,\delta (e)))\subseteq B(p(e),\varepsilon (p(e))).$
Then we define the funcion $\bar{\varepsilon }:E\rightarrow
(0,+\infty )$ as $\bar{\varepsilon }(e)=\min \{\delta
(e),\varepsilon (p(e))\}.$  Now suppose we have a subset
$A\subseteq B$ together with a strict local section
$s:A\rightarrow E$ of $p:$
$$\xymatrix{
{A} \ar@{^(->}[rr]^{inc_A} \ar[dr]_{s} & & B \\ & E \ar[ur]_p }$$
By Lemma \ref{walsh} there exist an open neighborhood $U\supseteq
A$ in $B$ and a map $s':U\rightarrow E$ satisfying that
$s'_{|A}\simeq s$ and for all $u\in U$ there exists $a_u\in A$
such that $\max \{d(s'(u),s(a_u)),d(u,a_u))\}<\bar{\varepsilon
}(s(a_u))$. If $inc_U:U\hookrightarrow B$ denotes the canonical
inclusion, then we want to prove that $p\circ s'$ and
$inc_U:U\rightarrow B$ are $\mathcal{W}$-close. Indeed, given
$u\in U$ we have that $d(s'(u),s(a_u))<\bar{\varepsilon
}(s(a_u))\leq \delta (s(a_u))$; in other words, $s'(u)\in
B(s(a_u),\delta (s(a_u)))$, so that
$$p(s'(u))\in p(B(s(a_u),\delta (s(a_u))))\subseteq B(a_u,\varepsilon (a_u))$$
\noindent that is, $p(s'(u))\in B(a_u,\varepsilon (a_u))$.
Moreover, $d(u,a_u)<\bar{\varepsilon }(s(a_u))\leq \varepsilon
(p(s(a_u)))=\varepsilon (a_u),$ so $u\in B(a_u,\varepsilon
(a_u)).$ This proves that $p\circ s'$ and $inc_U:U\rightarrow B$
are $\mathcal{W}$-close and therefore $p\circ s'\simeq inc_U.$
Finally, as $p$ is a fibration we can obtain a map
$s'':U\rightarrow E$ such that $p\circ s''=inc_U.$ Using this
argument for covers we conclude the proof.
\end{proof}

As a consequence of this theorem we obtain the next results:

\begin{corollary}
Let $X$ be a space having the homotopy type of a CW-complex
(equivalently, of an ANR). Then $\mbox{TC}(X)=\mbox{TC}_g(X).$
\end{corollary}

\begin{proof}
As $\mbox{TC}_g(-)$ y $\mbox{TC}(-)$ are homotopy invariants, we
can suppose that $X$ is an ANR space. In this case it is
well-known that $X\times X$ y $X^I$ are ANR spaces (see
\cite{Hu}). Then, applying Theorem \ref{chulo} above to the
fibration $\pi :X^I\rightarrow X\times X$ we conclude the proof.
\end{proof}

\begin{remark}
Another interesting example of TC-like numerical homotopy
invariant agreeing with its generalized version (i.e., using
general covers) is given by the higher analogs of topological
complexity $\mbox{TC}_m(-)$ \cite{R}.
\end{remark}

\bigskip
Finally, we also recover an already known result, prove by T.
Srinivasan \cite{Sr,Sr2}:

\begin{corollary}
Let $X$ be a space having the homotopy type of a connected
CW-complex  (equivalently, of a connected ANR space). Then
$\mbox{cat}(X)=\mbox{cat}_g(X).$
\end{corollary}

\begin{proof}
As $P(X)$ is an ANR space, one has just to follow for the path
fibration $p:P(X)\rightarrow X$ a similar argument to that done in
the previous corollary.
\end{proof}

\subsection{Generalized relative category}

We finish our study with the generalized relative category in the
sense of Doeraene-El Haouari. Recall from Corollary \ref{carac},
that for a given cofibration $i_X:A\hookrightarrow X$ where $X$ is
a normal space, $\mbox{relcat}(i_X)$ can be defined as the least
nonnegative integer $n$ (or infinity) such that $X$ admits a cover
by $n+1$ relatively sectional open subsets. This way, we can
straightforwardly define the generalized relative category, as
follows:

\begin{definition}
Let $i_X:A\hookrightarrow X$ be a cofibration. The
\emph{generalized relative category} of $i_X,$
$\mbox{relcat}_g(i_X),$ is defined as the least nonnegative
integer $n$ (or infinity) such that $X$ admits a cover by $n+1$
relatively sectional subsets.
\end{definition}

Generalized relative category is a homotopy invariant in the
following sense:

\begin{proposition}\label{g-hom-inv}
Consider the following homotopy commutative diagram, where $\alpha
$ and $\beta $ are homotopy equivalences and $i_X$, $i_{X'}$ are
cofibrations:
$$\xymatrix{
{A} \ar[r]^{\alpha }_{\simeq } \ar@{^{(}->}[d]_{i_X} & {A'} \ar@{^{(}->}[d]^{i_{X'}} \\
{X} \ar[r]_{\beta }^{\simeq } & {X'}  }$$ Then
$\mbox{relcat}_g(i_X)=\mbox{relcat}_g(i_{X'}).$
\end{proposition}

\begin{proof}
As $i_X$ is a cofibration we can suppose that the diagram is
strictly commutative and therefore we have  $\beta
:(X,A)\rightarrow (X',A')$ a homotopy equivalence of pairs where
$\beta _{|A}=\alpha :A\rightarrow A'.$ This way there exists
$\beta ':(X',A')\rightarrow (X,A)$ a homotopy inverse of $\beta $
and therefore we can take $G:(X\times I,A\times I)\rightarrow
(X,A)$ a homotopy of pairs such that $G(x,0)=x$ and $G(x,1)=(\beta
'\circ \beta )(x)$, for all $x\in X$.

Let $U'$ be a relatively sectional subset with respect to $i_{X'}$
and take $H:(U'\times I,A'\times I)\rightarrow (X',A')$ such that
$H(x,0)=x$ and $H(x,1)\in A'.$ If we define $U:=\beta ^{-1}(U')$,
then the homotopy of pairs $F:(U\times I,A\times I)\rightarrow
(X,A)$ given by
$$F(x,t)=\begin{cases}
G(x,2t), & 0\leq t\leq \frac{1}{2} \\
\beta '(H(\beta (x),2t-1)), & \frac{1}{2}\leq t\leq 1
\end{cases}$$
\noindent proves that $U$ is a relatively sectional subset with
respect to $i_X.$ Considering a cover argument we obtain the
inequality $\mbox{relcat}_g(i_X)\leq \mbox{relcat}_g(i_{X'}).$ The
other inequality is similarly proved.
\end{proof}

Now we want to prove that, when dealing among ANR spaces, we have
the equality $\mbox{relcat}(i_X)=\mbox{relcat}_g(i_X).$ In order
to prove this we begin by giving the next technical lemma:

\begin{lemma}\label{lematec1}
Let $i_X:A\hookrightarrow X$ be a cofibration between ANR spaces.
Then there exists a factorization:
$$\xymatrix{
{A} \ar@{^{(}->}[rr]^{i_X} \ar[dr]^{\simeq }_{j} & & {X} \\ &
{\widehat{A}} \ar@{>>}[ur]_p & }$$ \noindent where $\widehat{A}$
is an ANR space, $j$ a homotopy equivalence and $p$ a fibration.
\end{lemma}

\begin{proof}
The space $\widehat{A}$ is given from the classical mapping track
(or mapping cocylinder) construction
$\widehat{A}=L_{i_X}=\{(a,\alpha )\in A\times X^I:i_X(a)=\alpha
(0)\}$, in other words, from the pullback
$$\xymatrix{
{\widehat{A}}  \ar[d] \ar[r]  &  {X^I}  \ar[d]^{d_0} \\
{A} \ar[r]_{i_X} & {X} }$$ Such a construction gives us the
classical factorization of the map $i_X$ through a homotopy
equivalence followed by a fibration. Therefore it only remains to
check that $\widehat{A}$ is, actually, an ANR space. Indeed, we
first observe that $\widehat{A}$ is homeomorphic to $p^{-1}(A)$,
since $i_X$ is an inclusion. But the inclusion
$p^{-1}(A)\hookrightarrow X^I$ is a cofibration (see Theorem 12 in
\cite{Str}). As $p^{-1}(A)$ is a closed subspace of $X^I$, whose
corresponding inclusion is a cofibration, we have by Theorem 2.4
of chapt. VI and Theorem 3.2 of chapt. IV in \cite{Hu}, that
$p^{-1}(A)$ is an ANR space.
\end{proof}

Another lemma we need is related to certain kind of open covers.

\begin{lemma}\label{lematec2}
Let $i_X:A\hookrightarrow X$ be a cofibration with $X$ a normal
space. Also suppose $\{U_i\}_{i=0}^n$ an open cover of $X$ in such
a way that for any $i$ we have $A\subseteq U_i$ and there exists a
homotopy $H_i:U_i\times I\rightarrow X$ satisfying that
$H_i(x,0)=x,$ $H_i(x,1)\in A$ and $H_i(-,1)_{|A}\simeq 1_A.$ Then
$\mbox{relcat}(i_X)\leq n.$
\end{lemma}

\begin{proof}
Following a similar reasoning to that done in the proof of Theorem
\ref{chulo} we can consider a chain of inclusions $A\subseteq
A_i\subseteq \theta _i\subseteq B_i\subseteq U_i$ with $A_i,B_i$
closed subsets and $\theta _i$ open subsets, such that
$\{A_i\}_{i=0}^n$ covers $X.$ If $h_i:X\rightarrow [0,1]$ is the
Urysohn map given as $h_i(A_i)=\{1\}$ y $h_i(X\setminus \theta
_i)=\{0\}$ then the map $L_i:X\times I\rightarrow X$ defined by
$$L_i(x,t)=\begin{cases}
x, & x\in X\setminus B_i \\
H_i(x,t\cdot h_i(x)), & x\in U_i
\end{cases}$$
\noindent is clearly continuous. Taking $L=(L_0,...,L_n):X\times
I\rightarrow X^{n+1}$ we have a homotopy $L:\Delta _{n+1}\simeq
t_n\circ f,$ where $f:=L(-,1):X\rightarrow T^n(i_X).$ What is
more, taking a homotopy $G_i:A\times I\rightarrow A$ satisfying
$G_i:1_A\simeq H_i(-,1)_{|A}$ it is not difficult to check that
$G:=(G_0,...,G_n):A\times I\rightarrow A^{n+1}\subseteq T^n(i_X)$
gives a homotopy $G:\tau _n\simeq f\circ i_X.$ These facts prove
that the following diagram commutes, up to homotopy:
$$\xymatrix{
{A} \ar[rr]^{\tau _n} \ar@{^{(}->}[d]_{i_X} & & {T^n(i_X)}
\ar@{^{(}->}[d]^{t_n} \\
{X} \ar[urr]^f \ar[rr]_{\Delta _{n+1}} & & {X^{n+1}}   }$$
Therefore, by Theorem \ref{DH} we have the inequality
$\mbox{relcat}(i_X)\leq n.$
\end{proof}

\begin{remark}
Observe, that the converse in the lemma above also holds, that is,
if $\mbox{relcat}(i_X)\leq n$, then there exists an open cover as
in the statement of the lemma. Therefore we actually have another
equivalent characterization of $\mbox{relcat}(i_X)$ (compare with
Corollary \ref{carac} (iv)).
\end{remark}

Now we can state and prove the main result.

\begin{theorem}\label{chulo2}
Let $i_X:A\hookrightarrow X$ be a cofibration between ANR spaces.
Then $\mbox{relcat}(i_X)=\mbox{relcat}_g(i_X).$
\end{theorem}

\begin{proof}
We only have to check the inequality $\mbox{relcat}(i_X)\leq
\mbox{relcat}_g(i_X).$ Suppose we have an arbitrary subset
$U\subseteq X$ such that $A\subseteq U$ and a homotopy of pairs
$H:(U\times I,A\times I)\rightarrow (X,A)$ satisfying $H(x,0)=x$
and $H(x,1)\in A$, for all $x\in X$. Considering a factorization
$i_X=p\circ j$ through the space $\widehat{A}$ as in Lemma
\ref{lematec1} above, we can define a map $s:U\rightarrow
\widehat{A}$ as the composite
$$U\stackrel{H(-,1)}{\longrightarrow }A\stackrel{j}{\longrightarrow }\widehat{A}$$
Clearly $p\circ s\simeq inc_U$ and $s_{|A}\simeq j,$ where
$inc_U:U\hookrightarrow X$ stands for the natural inclusion.
Observe the, since $p$ is a fibration, we can suppose without
losing any generality that $p\circ s=inc_U$ and $s_{|A}\simeq j.$
Following the proof in Theorem \ref{chulo} one has an open
neighborhood $V\supseteq U$ in $X$ and a map $\sigma :V\rightarrow
\widehat{A}$ such that $p\circ \sigma =inc_V$ and $\sigma
_{|U}\simeq s.$ In particular, one also has
$$\sigma _{|A}=(\sigma _{|U})_{|A}\simeq s_{|A}\simeq j$$
Besides that, if $j':\widehat{A}\rightarrow A$ denotes an homotopy
inverse of $j$, then
$$i_X\circ j'\circ \sigma =p\circ j\circ j'\circ \sigma \simeq
p\circ \sigma =inc_V$$ In other words, there exists a homotopy
$G:V\times I\rightarrow X$ such that $G:inc_V\simeq i_X\circ
j'\circ \sigma $. Obviously, $G(x,1)\in A$ and
$G(-,1)_{|A}=j'\circ (\sigma _{|A})\simeq j'\circ j\simeq 1_A.$
Taking into account an argument by covers and using Lemma
\ref{lematec2}, we conclude the proof.
\end{proof}

As a consequence of the above theorem we can compare the monoidal
topological complexity $\mbox{TC}^M(X)$ and its generalized
version. Recall that the \emph{monoidal topological complexity} of
a space $X$ (see \cite{I-S}) is defined as the least nonnegative
integer $n$ such that $X\times X$ admits an open cover
$\{U_i\}_{i=0}^n$ such that each $U_i$ contains $\Delta
_X(X)=\{(x,x):x\in X\}$ and admits a strict local section
$s_i:U_i\rightarrow X^I$ of the path fibration $\pi
_X:X^I\rightarrow X\times X$ satisfying $s_i(x,x)=c_x$ (here,
$c_x$ denotes the constant path in $X$ at the point $x$). Unlike
the topological complexity, $\mbox{TC}^M(-)$ is not a homotopy
invariant. For instance, consider $X$ the subspace of
$\mathbb{R}^2$
$$X=\bigcup
_{n=0}^{\infty }L_n$$ \noindent where $L_n=\{(t,\frac{t}{n}):t\in
[0,1]\}$ is the segment joining $(0,0)$ with $(1,\frac{1}{n})$ for
$n\geq 1$ and $L_0=[0,1]\times \{0\}$. Clearly $X$ is a
contractible space. Nevertheless, on the other hand we have that
$\mbox{TC}^M(X)\neq 0.$ Indeed, if $\mbox{TC}^M(X)=0$ and
$s:X\times X\rightarrow X^I$ is a global continuous section of the
path fibration $\pi :X^I\rightarrow X\times X$ satisfying
$s(x,x)=c_x$, then for $n\geq 1$ we can take the sequence of
points $x_n=(1,\frac{1}{n})$ and $x_0=(1,0).$ Obviously,
$(x_n,x_0)\rightarrow (x_0,x_0)$ in $X\times X,$ when
$n\rightarrow \infty $ and by continuity we have that
$s(x_n,x_0)\rightarrow s(x_0,x_0)=c_{x_0}.$ This is not possible,
as for all $n\geq 1$, $s(x_n,x_0)$ is a path from $x_n$ to $x_0$
that must pass through the origin $(0,0)$. This example also shows
that, in general, $\mbox{TC}(X)\neq \mbox{TC}^M(X)$. We point out
that this does not refute Iwase-Sakai's conjecture \cite{I-S}
which asserts that $\mbox{TC}(X)=\mbox{TC}^M(X)$ for any locally
finite simplicial complex $X$.

As we are going to see, restricting to a suitable large class of
spaces, monoidal topological complexity becomes a homotopy
invariant. Recall that a locally equiconnected space is a
topological space $X$ in which the diagonal map $\Delta
_X:X\rightarrow X\times X$ is a cofibration. This class of spaces
is sufficiently large. For instance, it is known that CW-complexes
and ANR spaces belong to such a class.

\begin{proposition}\cite[Th. 12]{C-G-V}\label{crucial}
If $X$ is a locally equiconnected space, then
$\mbox{TC}^M(X)=\mbox{relcat}(\Delta _X).$ In particular,
$\mbox{TC}^M(-)$ is a homotopy invariant when we restrict to
locally equiconnected spaces.
\end{proposition}

If we define the \emph{generalized monoidal topological
complexity} of a locally equiconnected space $X$ as
$\mbox{TC}^M_g(X):=\mbox{relcat}_g(\Delta _X)$, then we have the
following corollary of Theorem \ref{chulo2}:

\begin{corollary}\label{monoidal}
If $X$ is a locally equiconnected space having the homotopy type
of a CW-complex (equivalently, of an ANR-space), then
$\mbox{TC}^M(X)=\mbox{TC}^M_g(X).$
\end{corollary}

\begin{proof}
Just take into account Proposition \ref{g-hom-inv}, Proposition
\ref{crucial} and Theorem \ref{chulo2} above.
\end{proof}

\begin{remark}
There is an alternative definition of $\mbox{TC}^M_g(X)$ for a non
necessary locally equiconnected space. We can take subsets
$U\subseteq X\times X$ such that $\Delta _X(X)\subseteq U$
together with a strict local section $s:U\rightarrow X^I$ of $\pi
_X:X^I\rightarrow X\times X$ satisfying $s(x,x)=c_x,$ for all
$x\in X.$ Then, define $\mbox{TC}^M_g(X)$ as the least nonnegative
integer $n$ (or infinity) such that $X\times X$ admits a cover of
this kind of subsets. With this alternative definition we do not
know if $\mbox{TC}^M_g(X)=\mbox{relcat}_g(\Delta _X)$ when $X$ is
a locally equiconnected space. What we have is the inequality
$\mbox{relcat}_g(\Delta _X)\leq \mbox{TC}^M_g(X)$. Indeed, suppose
a subset $U\subseteq X\times X$ verifying the properties above
mentioned. Then $H:U\times I\rightarrow X\times X$, defined as
$H(x,y,t):=(s(x,y)(t),y)$ verifies $H(x,y,0)=(x,y),$
$H(x,y,1)=(y,y)\in \Delta _X(X)$ and $H(x,x,t)=(x,x).$ In other
words, $U$ is a relatively sectional subset with respect to the
cofibration $\Delta _X.$ Considering covers one can
straightforwardly deduce that $\mbox{relcat}_g(\Delta _X)\leq
\mbox{TC}^M_g(X)$. In this case, concerning Corollary
\ref{monoidal}, the more we can say is: ``If $X$ is an ANR space,
then $\mbox{TC}^M(X)=\mbox{TC}^M_g(X)$". Indeed, combining the
inequality $\mbox{relcat}_g(\Delta _X)\leq \mbox{TC}^M_g(X)$ with
Proposition \ref{crucial} and Theorem \ref{chulo2} we obtain
$$\mbox{TC}^M(X)=\mbox{relcat}(\Delta _X)=\mbox{relcat}_g(\Delta
_X)\leq \mbox{TC}^M_g(X)$$ \noindent concluding that
$\mbox{TC}^M(X)=\mbox{TC}^M_g(X).$
\end{remark}


\begin{thebibliography}{99}

\bibitem{C-G-V}  J.G. Carrasquel-Vera, J.M. García-Calcines and
L. Vandembroucq,  Relative category and monoidal topological
complexity. \emph{Topology Appl.} 171 (2014), 41--53.

\bibitem{D-H} J.P. Doeraene and M. El Haouari. Up to one
approximations of sectional category and topological complexity.
\emph{Topology Appl.}, 160(5) (2013), 766--783.

\bibitem{F-H} E. Fadell and S. Husseini, Relative category, products and
coproducts, \emph{Rend. Sem. Mat. Fis. Milano} 64 (1996), 99--115.

\bibitem{F} M. Farber, Topological complexity of motion planning, \emph{Discrete
Comput. Geom.} 29 (2003) 211--221.

\bibitem{C-V}{J. M. Garc\'{\i}a Calcines and L.
Vandembroucq}. Weak sectional category, \textit{Journal of the
London Math. Soc.} \textbf{82}(3) (2010), 621-642.

\bibitem{Hu} S.T. Hu, Theory of Retracts, Wayne State Univ. Press, 1965.

\bibitem{I-S} N. Iwase, M. Sakai, Topological complexity is a fibrewise L-S
category, \emph{Topol. Appl.} 157 (2010) 10--21. (Erratum to
''Topological complexity is a fibrewise LS-category", \emph{Topol.
Appl.} 159 (2012) 2810--2813).

\bibitem{J} I.M. James, On category in the sense of
Lusternik-Schnirelmann, \emph{Topology} 17 (1978) 331--348.

\bibitem{R} Y.B. Rudyak. On higher analogs of topological
complexity, Topology Appl. 157 (2010) 916--920 (erratum in
Topology Appl. 157 (2010), 1118).

\bibitem{Str} A. Str{\o}m. Note on cofibrations II. \emph{Math. Scand.}
22 (1968), 130--142.


\bibitem{Sr}
{Srinivasan, Tulsi. On the Lusternik-Schnirelmann category of
Peano continua. \emph{Topology Appl.}} 160 (2013), no. 13,
1742--1749.

\bibitem{Sr2}
{Srinivasan, Tulsi. The Lusternik-Schnirelmann category of metric
spaces. \emph{Topology Appl.}} 167 (2014), 87--95.

\bibitem{Sch} {\sc A. Schwarz}. \emph{The genus of a fiber space}. A.M.S. Transl. 55 (1966),
49-140

\bibitem{W}
{J.J. Walsh. Dimension, cohomological dimension and cell-like
mappings, \textit{Lecture Notes Math.}} vol.870, Springer, 1981.

\bibitem{War} {G. Warner}. Topics in Topology and Homotopy
Theory.  http://www.math.washington.edu/~warner/


\end{thebibliography}
\end{document}